\documentclass[11pt,a4paper]{amsart}
\usepackage{amsthm,amsfonts,amsmath,hyperref}
\usepackage[numbers,sort&compress]{natbib}
\usepackage[lmargin=34mm,rmargin=34mm,tmargin=33mm,bmargin=33mm,footskip=10mm]{geometry}
\allowdisplaybreaks

\newcommand{\thmlabel}[1]{\label{thm:#1}}
\newcommand{\thmref}[1]{Theorem~\ref{thm:#1}}
\newcommand{\lemlabel}[1]{\label{lem:#1}}
\newcommand{\lemref}[1]{Lemma~\ref{lem:#1}}
\newcommand{\proplabel}[1]{\label{prop:#1}}

\theoremstyle{plain}
\newtheorem{theorem}{Theorem}
\newtheorem{lemma}[theorem]{Lemma}

\newtheorem{proposition}[theorem]{Proposition}
\theoremstyle{definition}
\newtheorem{definition}[theorem]{Definition}

\newtheorem*{WLHC}{Weak List Hadwiger Conjecture}
\newtheorem*{LHC}{List Hadwiger Conjecture}
\newtheorem*{HC}{Hadwiger Conjecture}
\newtheorem*{WHC}{Weak Hadwiger Conjecture}

\newcommand{\plus}[2]{\ensuremath{{#1}^{{#2}+\alpha}}}

\newcommand{\M}{\ensuremath{\mathcal{M}}}

\renewcommand{\baselinestretch}{1.105}
\begin{document}

\title{Thomassen's Choosability Argument Revisited}

\author{David R. Wood}
\address{\newline Department of Mathematics and Statistics
\newline The University of Melbourne
\newline Melbourne, Australia}
\email{woodd@unimelb.edu.au}
\thanks{David Wood is a QEII Research Fellow supported by the
  Australian Research Council.}

\author{Svante Linusson}
\address{\newline Department of Mathematics
\newline KTH--Royal Institute of Technology
\newline Stockholm, Sweden}
\email{linusson@math.kth.se}
\thanks{Svante Linusson is a Royal Swedish Academy of Sciences
  Research Fellow supported by a grant from the Knut and Alice
  Wallenberg Foundation.}


\date{\today}

\begin{abstract} 
Thomassen (1994) proved that every planar graph is $5$-choosable. This
result was generalised by {\v{S}}krekovski (1998) and He et
al.~(2008), who proved that every $K_5$-minor-free graph is
$5$-choosable. Both proofs rely on the characterisation of
$K_5$-minor-free graphs due to Wagner (1937). This paper proves the
same result without using Wagner's structure theorem or even planar
embeddings. Given that there is no structure theorem for graphs with
no $K_6$-minor, we argue that this proof suggests a possible approach
for attacking the Hadwiger Conjecture. 
\end{abstract}

\maketitle

%%%%%%%%%%%%%%%%%%%%%%%%%%%%%%%%%%%%%%%%%%%%%%%%%%%%%%%%%%%%%%
\section{Introduction}
\label{Intro}
%%%%%%%%%%%%%%%%%%%%%%%%%%%%%%%%%%%%%%%%%%%%%%%%%%%%%%%%%%%%%%

In 1943, \citet{Hadwiger43} made the following conjecture, which is
widely considered to be one of the most important open problems in
graph theory\footnote{All graphs in this paper are undirected, simple
  and finite.  We employ standard graph-theoretic terminology and
  notation \citep{Diestel05}.}; see \citep{Toft-HadwigerSurvey96} for
a survey.

\begin{HC} Every $K_t$-minor-free graph is
  $(t-1)$-colourable.  \end{HC}

The Hadwiger Conjecture is true for $t\leq6$ \citep{RSST97,RST-Comb93}
and unsolved for $t\geq7$. In general, $ct\sqrt{\log t}$ is the best
known upper bound on the chromatic number of $K_t$-minor-free graphs,
for some constant $c$ \citep{Kostochka82,Thomason84}. This result is
proved as follows. A graph $G$ is \emph{$d$-degenerate} if every
subgraph of $G$ has a vertex of degree at most $d$. Every
$d$-degenerate graph is $(d+1)$-colourable---choose a vertex $v$ of
degree at most $d$, apply induction to $G-v$, and colour $v$ with one
of the colours not present in its neighbourhood. \citet{Kostochka82}
and \citet{Thomason84} independently proved that every
$K_t$-minor-free graph is $ct\sqrt{\log t}$-degenerate, and is thus
$ct\sqrt{\log t}$-colourable. The following conjecture remains
unsolved.

\begin{WHC} There is a constant $c$ such that every $K_t$-minor-free
  graph is $ct$-colourable.  \end{WHC}

There are (at least) two major obstacles to overcome in a proof of the
Hadwiger Conjecture or the Weak Hadwiger Conjecture:
\begin{itemize}
\item There are $K_t$-minor-free graphs with minimum degree
  $ct\sqrt{\log t}$ for some constant $c$. Therefore the above
  degeneracy-based algorithm fails.
\item For $t\geq 6$, there is no known precise structural
  characterisation of $K_t$-minor-free graphs (and even for $t=6$ the
  situation seems hopeless).
\end{itemize}
This paper suggests a possible approach around these two obstacles.

A \emph{list-assignment} of a graph $G$ is a function $L$ that assigns
to each vertex $v$ of $G$ a set $L(v)$ of colours. $G$ is
\emph{$L$-colourable} if there is a colouring of $G$ such that the
colour assigned to each vertex $v$ is in $L(v)$. $G$ is
\emph{$k$-choosable} if $G$ is $L$-colourable for every
list-assignment $L$ with $|L(v)|\geq k$ for each vertex $v$ of $G$. If
$G$ is $k$-choosable then $G$ is also $k$-colourable---just use the
same set of $k$ colours for each vertex. Also note that every
$d$-degenerate graph is $(d+1)$-choosable. See
\citep{Woodall-ListColouringSurvey} for a survey on list colourings.
\citet{KawaMohar-JCTB07} made the following conjecture:

\begin{WLHC} There is a constant $c$ such that every $K_t$-minor-free
  graph is $ct$-choosable.  \end{WLHC}

\citet{KawaMohar-JCTB07} wrote that they believe the Weak List
Hadwiger Conjecture holds for
$c=\frac{3}{2}$. \citet{Wood-Contractibility} conjectured it with
$c=1$.

\begin{LHC} Every $K_t$-minor-free graph is $t$-choosable.  \end{LHC}

For $t\in\{2,3,4\}$, every $K_t$-minor-free graph is
$(t-2)$-degenerate, and thus is $(t-1)$-choosable. Now consider the
$t=5$ case. \citet{Thomassen-JCTB94} proved that every planar graph is
$5$-choosable, and \citet{Voigt-DM93} constructed planar graphs that
are not 4-choosable. Thomassen's result was generalised by
\citet{Skrekovski-DM98} and \citet{HMS-DM08} as follows:

\begin{theorem}[\citep{Skrekovski-DM98,HMS-DM08}] 
  \thmlabel{Choose5} Every $K_5$-minor-free graph is $5$-choosable.
\end{theorem}

One feature of Thomassen's proof is that it does not depend on the
degeneracy of planar graphs. Thus list colourings provide a potential
route around the first obstacle above. See
\citep{KawaMohar-JCTB07,Wood-Contractibility} for more concrete
examples of this idea. The second obstacle remains. In particular,
Thomassen's proof relies heavily on the structure of planar graphs, as
do the proofs of \thmref{Choose5}, both of which employ the structural
characterisation of $K_5$-minor graphs in terms of planar graphs due
to \citet{Wagner37}. The main contribution of this paper is to prove
\thmref{Choose5} without using Wagner's characterisation---even
without planar embeddings. Given that there is no precise structure
theorem for $K_t$-minor-free graphs for $t\geq6$, we consider this a
first step towards proving the (Weak) List Hadwiger Conjecture for
$t\geq6$.

%%%%%%%%%%%%%%%%%%%%%%%%%%%%%%%%%%%%%%%%%%%%%%%%%%%%%%%%%%%%%%
\section{Proof of \thmref{Choose5}}
%%%%%%%%%%%%%%%%%%%%%%%%%%%%%%%%%%%%%%%%%%%%%%%%%%%%%%%%%%%%%%

Our proof of \thmref{Choose5} is inspired by Thomassen's proof for
planar graphs. This remarkable inductive argument allows two adjacent
vertices on the outerface to be precoloured (that is, have a list of
one colour), the remaining vertices on the outerface have a list of
three colours, and the other vertices have a list of five colours. The
dependence on the outerface is an obstacle to generalising Thomassen's
proof, and motivates the following definition.

\begin{definition}
  Let \M\ be a minor-closed class of graphs. Let $B\subseteq V(G)$ for
  some graph $G\in\M$. Let $\plus{G}{B}$ be the graph obtained from
  $G$ by adding a new vertex $\alpha$ adjacent to each vertex in
  $B$. Then $B$ is an \emph{\M-boundary} of $G$ if $\plus{G}{B}$ is
  also in \M.
\end{definition}

This definition generalises the outerface, since $\plus{G}{B}$ is
planar if and only if all the vertices in $B$ are on the outerface of
some planar embedding of $G$.

\begin{lemma}
  \lemlabel{Contract} Let \M\ be a minor-closed class of graphs.  Let
  $B$ be an \M-boundary of some graph $G\in\M$.  Let $v\in B$.  Then
  $C:=(B\setminus\{v\})\cup N_G(v)$ is an $\M$-boundary of $H:=G-v$.
\end{lemma}

\begin{proof}
  Observe that $\plus{H}{C}$ is isomorphic to the graph obtained from
  $\plus{G}{B}$ by contracting the edge $v\alpha$. Since
  $\plus{G}{B}\in\M$ and \M\ is minor-closed, $\plus{H}{C}$ is also in
  \M. That is, $C$ is an \M-boundary of $H$.
\end{proof}

The following lemma is a corollary of a more general result by
\citet{Mader-DM88}; we include the following simple proof for
completeness.

\begin{lemma}
  \lemlabel{Contractible} Let $v$ be a vertex in a $2$-connected graph
  $G$. Then $G/vw$ is $2$-connected for some edge $vw$ incident to
  $v$.
\end{lemma}

\begin{proof}
  Suppose on the contrary that $G/vw$ is not $2$-connected for each
  edge $vw$ incident to $v$; thus $\{v,w\}$ is a cut set. Choose such
  an edge $vw$ to minimise the order of a smallest component $H$ of
  $G-\{v,w\}$. Since $G$ is $2$-connected, $v$ has a neighbour $x$ in
  $H$. Thus $G-\{v,x\}$ contains a component that is a proper subgraph
  of $H$, which contradicts the choice of $vw$.
\end{proof}

Let $x,y,z$ be distinct vertices in a graph $G$. A \emph{$K_3$-minor
  rooted} at $x,y,z$ consists of three connected subgraphs $X,Y,Z$ of
$G$ that are pairwise disjoint and pairwise adjacent, such that $x\in
V(X)$, $y\in V(Y)$ and $z\in V(Z)$.  See
\citep{Wollan-JGT08,JorgKawa-JGT07,Kawa-DM04} for more on rooted
minors. A vertex $v$ of $G$ is \emph{good} (with respect to $x,y,z$)
if at least two of $x,y,z$ are in the same component of $G-v$,
otherwise $v$ is \emph{bad}. Note that if $v$ is a vertex in a
$2$-connected graph $G$, then $G-v$ is connected, and all the vertices
in $\{x,y,z\}\setminus\{v\}$ are in the same component of $G-v$; thus
at least two of $x,y,z$ are in one component of $G-v$. That is, every
vertex is good in a $2$-connected graph.

% all the vertices in $\{x,y,z\}\setminus\{v\}$ are in distinct
% components of $G-v$.

\begin{lemma}
  \lemlabel{GeneralRooted} Let $x,y,z$ be distinct vertices in a graph
  $G$. Then $G$ has a $K_3$-minor rooted at $x,y,z$ if and only if
  every vertex in $G$ is good.
\end{lemma}

\begin{proof}
  $(\Longrightarrow)$ Let $X,Y,Z$ be the branch sets of a $K_3$-minor
  rooted at $x,y,z$. Let $v\in V(G)$. Without loss of generality,
  $v\not\in X\cup Y$. Since $G[X\cup Y]$ is connected, $x$ and $y$ are
  in the same component of $G-v$. Thus $v$ is good.

  $(\Longleftarrow)$ We proceed by induction. Let $x,y,z$ be distinct
  vertices in a graph $G$ in which every vertex is good. If $|V(G)|=3$
  and $G\not\cong K_3$, then without loss of generality, $G$ is a
  subgraph of the path $(x,y,z)$, implying $y$ is bad. Thus, if
  $|V(G)|=3$ then $G\cong K_3$, and we are done. Now assume that
  $|V(G)|\geq4$.

  First suppose that $G$ is disconnected. If $x$, $y$ and $z$ are all
  in the same component $H$ of $G$, then by induction, $H$ and hence
  $G$ has a $K_3$-minor rooted at $x,y,z$. Otherwise some component
  contains at most one of $x,y,z$, say $x$. Then $y$ and $z$ are both
  bad.

  Now assume that $G$ is connected. Suppose that $G$ contains a
  cut-vertex $v$. Since $v$ is good, at least two of $x,y,z$, say $x$
  and $y$, are in the same component of $G-v$. Let $w$ be a neighbour
  of $v$ in a component of $G-v$ not containing $x$ and $y$.  Let $G'$
  be the graph obtained from $G$ by contracting $vw$ into a vertex
  $v'$. We may consider $x,y,z$ to be vertices of $G'$. (It is
  possible that $w=z$ and $v'=z$.)\ In $G'$, the vertex $v'$ is good
  since $x$ and $y$ remain in the same component of $G'-v'$. If some
  other vertex in $G'$ is bad, then it would be bad in $G$. Thus every
  vertex in $G'$ is good. By induction, $G'$ and hence $G$ contains a
  $K_3$-minor rooted at $x,y,z$.

  Now assume that $G$ is $2$-connected. Choose $v\in
  V(G)\setminus\{x,y,z\}$. By \lemref{Contractible}, $G/vw$ is
  $2$-connected for some edge $vw$ incident to $v$. Thus every vertex
  is good in $G/vw$.  Since $x,y,z$ are distinct vertices in $G/vw$,
  by induction, $G/vw$ and hence $G$ has a $K_3$-minor rooted at
  $x,y,z$.
\end{proof}

A graph $G$ is said to contain \emph{every rooted $K_3$-minor} if $G$
contains a $K_3$-minor rooted at $x,y,z$ for all distinct $x,y,z\in
V(G)$.

% ; see \citep{???} for more on rooted minors.

 \begin{proposition}
   \proplabel{Rooted} A graph $G$ contains every rooted $K_3$-minor if
   and only if $G$ is $2$-connected.
 \end{proposition}

\begin{proof}
  Since every vertex is good in a $2$-connected graph, by
  \lemref{GeneralRooted}, a $2$-connected graph contains every rooted
  $K_3$-minor. For the converse, let $G$ be a graph that contains
  every rooted $K_3$-minor. If $G$ is disconnected, then there is no
  $K_3$-minor rooted at $x,y,z$, whenever $x$ and $y$ are in distinct
  components. Hence $G$ is connected. If $G$ has a cut-vertex $x$,
  then $G$ contains no $K_3$-minor rooted at $x,y,z$, whenever $y$ and
  $z$ are in distinct components of $G-x$. Hence $G$ is $2$-connected.
\end{proof}

% Every $2$-connected graph contains every rooted $K_3$-minor.

% \begin{proof}
%   We proceed by induction.  Let $x,y,z$ be distinct vertices in a
%   $2$-connected graph $G$.  If $|V(G)|=3$ then $G\cong K_3$ and we
%   are done.  Now assume that $|V(G)|\geq4$.  Choose $v\in
%   V(G)\setminus\{x,y,z\}$. By \lemref{Contractible}, $G/vw$ is
%   $2$-connected for some edge $vw$ incident to $v$. Since $x,y,z$
%   are distinct vertices in $G/vw$, by induction, $G/vw$ and hence
%   $G$ has a $K_3$-minor rooted at $x,y,z$.
% \end{proof}

Let $G_1$ and $G_2$ be subgraphs of a graph $G$, such that $G=G_1\cup
G_2$ and $V(G_1)\setminus V(G_2)\neq\emptyset$ and $V(G_2)\setminus
V(G_1)\neq\emptyset$.  In particular, there is no edge between
$V(G_1)\setminus V(G_2)$ and $V(G_2)\setminus V(G_1)$.  Then
$\{G_1,G_2\}$ is a \emph{separation} of order $|V(G_1)\cap V(G_2)|$.

\thmref{Choose5} is a consequence of the following lemma (with
$A=B=\emptyset$).

\begin{lemma}
  Let $\M$ be the class of $K_5$-minor-free graphs. Let $G\in\M$.  Let
  $A\subseteq B\subseteq V(G)$, such that $A$ is a clique, and $B$ is
  an $\M$-boundary of $G$. Let $L$ be a list-assignment of $G$ such
  that:
  \begin{itemize}
  \item $|L(x)|=1$ for each vertex $x\in A$,
  \item $L(x)\neq L(y)$ for distinct $x,y\in A$,
  \item $|L(x)|\geq 3$ for each vertex $x\in B\setminus A$,
  \item $|L(x)|\geq 5$ for each vertex $x\in V(G)\setminus B$.
  \end{itemize}
  Then $G$ is $L$-colourable.
\end{lemma}

\begin{proof} 
  Let $(G,A,B,L)$ be a counterexample with $|V(G)|$ minimum, and then
  with $|A|$ maximum. Clearly $|V(G)|\geq4$.

  \textbf{\boldmath Case 1. $B=\emptyset$:} Choose $v\in V(G)$. Then
  $\plus{G}{\{v\}}$, which is obtained from $G$ by adding a new vertex
  $\alpha$ adjacent to $v$, is $K_5$-minor-free. Let $L'(v):=\{c\}$
  for some colour $c\in L(v)$. Let $L'(x):=L(x)$ for every other
  vertex $x$. By the choice of $(G,A,B,L)$, the instance
  $(G,\{v\},\{v\},L')$ is not a counterexample, and $G$ is
  $L$-colourable. Now assume that $B\neq\emptyset$.

  \textbf{\boldmath Case 2. $A=\emptyset$:} Choose $v\in B$. Let
  $L'(v):=\{c\}$ for some colour $c\in L(v)$. Let $L'(x):=L(x)$ for
  every other vertex $x$. Again $(G,\{v\},B,L')$ is not a
  counterexample, and $G$ is $L$-colourable. Now assume that
  $A\neq\emptyset$.

  \textbf{\boldmath Case 3. $G$ is not connected:} Then $G$ contains a
  separation $\{G_1,G_2\}$ with $V(G_1\cap G_2)=\emptyset$.  Since $A$
  is a clique, without loss of generality, $A\subseteq V(G_1)$. Let
  $B_i:=B\cap V(G_i)$. Then $B_i$ is an $\M$-boundary of $G_i$ (since
  $\plus{G_i}{B_i}\subseteq\plus{G}{B}$).  Define $L_i(x):=L(x)$ for
  each vertex $x$ in $G_i$. Hence $(G_1,A,B_1,L_1)$ is not a
  counterexample, and $G_1$ is $L_1$-colourable. Also
  $(G_2,\emptyset,B_2,L_2)$ is not a counterexample, and $G_2$ is
  $L_2$-colourable. Hence $G$ is $L$-colourable. Now assume that $G$
  is connected.

  \textbf{\boldmath Case 4. $G$ contains a cut-vertex $v\in B$:} Then
  $G$ contains a separation $\{G_1,G_2\}$ with $V(G_1\cap G_2)=\{v\}$.
  Since $A$ is a clique, without loss of generality, $A\subseteq
  V(G_1)$. Let $B_i:=B\cap V(G_i)$. Then $B_i$ is an $\M$-boundary of
  $G_i$. Define $L_1(x):=L(x)$ for each vertex $x$ in $G_1$. Thus
  $(G_1,A,B_1,L_1)$ is not a counterexample, and $G_1$ is
  $L_1$-colourable. Let $L_2(v):=\{c\}$ where $v$ is coloured $c$ in
  $G_1$.  Let $L_2(x):=L(x)$ for every other vertex $x$ in $G_2$. Then
  $(G_2,\{v\},B_2,L_2)$ is not a counterexample (since $v\in B_2$),
  and $G_2$ is $L_2$-colourable. Hence $G$ is $L$-colourable (since
  $v$ receives the same colour in $G_1$ and in $G_2$, and each edge of
  $G$ is in $G_1$ or $G_2$). Now assume that $G-v$ is connected for
  every vertex $v\in B$.

  \textbf{\boldmath Case 5. $G$ contains a cut-vertex $v$ separating
    two vertices in $B$:} Then $G$ contains a separation $\{G_1,G_2\}$
  with $V(G_1\cap G_2)=\{v\}$, such that $B\cap V(G_1-v)\neq\emptyset$
  and $B\cap V(G_2-v)\neq\emptyset$.  Since $A$ is a clique, without
  loss of generality, $A\subseteq V(G_1)$. Let $B_1:=B\cap
  V(G_1)$. Then $B_1$ is an $\M$-boundary of $G_1$. Define
  $L_1(x):=L(x)$ for each vertex $x$ in $G_1$. Thus $(G_1,A,B_1,L_1)$
  is not a counterexample, and $G_1$ is $L_1$-colourable.

  Since $G$ is connected, $G_1$ is connected. Let $B_2:=(B\cap
  V(G_2))\cup\{v\}$. Then $B_2$ is an \M-boundary of $G_2$, since
  \plus{G_2}{B_2} is a minor of \plus{G}{B} obtained by contracting
  $G_1$ into $v$ (since $\alpha$ has a neighbour in $G_1-v$).  Let
  $L_2(v):=\{c\}$ where $v$ is coloured $c$ in $G_1$.  Let
  $L_2(x):=L(x)$ for every other vertex $x$ in $G_2$. Then
  $(G_2,\{v\},B_2,L_2)$ is not a counterexample (since $v\in B_2$),
  and $G_2$ is $L_2$-colourable. Hence $G$ is $L$-colourable. Now
  assume that $G$ contains no cut-vertex separating two vertices in
  $B$.

  \textbf{\boldmath Case 6. $G$ contains a cut-set $\{v,w\}$
    separating two vertices in $B$, where $v\in B$:} Thus $G$ has a
  separation $\{G_1,G_2\}$ with $V(G_1)\cap V(G_2)=\{v,w\}$ and $B\cap
  V(G_1-\{v,w\})\neq\emptyset$ and $B\cap
  V(G_2-\{v,w\})\neq\emptyset$.

  Suppose that $vw\not\in E(G)$. We claim that adding the edge $vw$
  creates no $K_5$-minor in $\plus{G}{B}$. Let $G'$ be the graph
  obtained from $\plus{G}{B}$ by adding the edge $vw$. Let $H$ be a
  4-connected minor in $G'$. Since $\{\alpha,v,w\}$ is a separator in
  $G'$, there are no two branch sets of $H$, with one contained in
  $V(G_1)\setminus V(G_2)$, and the other contained in
  $V(G_2)\setminus V(G_1)$. Thus, without loss of generality, every
  branch set of $H$ intersects $G_1$. By Cases 4 and 5, neither $v$
  nor $w$ are cut-vertices in $G$. Thus there is $vw$-path $P$ in
  $G_2$. Hence the edge $vw$ in our $H$-minor can be replaced by $P$
  to obtain an $H$-minor in $\plus{G}{B}$ (without $vw$). Since $K_5$
  is 4-connected and $\plus{G}{B}$ is $K_5$-minor-free, $G'$ is also
  $K_5$-minor-free. That is, adding $vw$ does not create a $K_5$-minor
  in $\plus{G}{B}$ (and also not in $G$). Since $A$ is a clique,
  adding $vw$ does not break any of the assumptions in the lemma. Now
  assume that $vw\in E(G)$.

  Since $A$ is a clique, without loss of generality, $A\subseteq
  V(G_1)$. Let $B_1:=B\cap V(G_1)$. Then $B_1$ is an $\M$-boundary of
  $G_1$. Define $L_1(x):=L(x)$ for each vertex $x$ in $G_1$. Thus
  $(G_1,A,B_1,L_1)$ is not a counterexample, and $G_1$ is
  $L_1$-colourable.

  Since $G$ is connected and $v$ is not a cut-vertex, $G_1-v$ is
  connected. Let $B_2:=(B\cap V(G_2))\cup\{w\}$. Then $B_2$ is an
  \M-boundary of $G_2$, since \plus{G_2}{B_2} is a minor of
  \plus{G}{B} obtained by contracting $G_1-v$ into $w$ (since $\alpha$
  has a neighbour in $G_1-v$). Let $L_2(v):=\{c\}$ where $v$ is
  coloured $c$ in $G_1$.  Let $L_2(w):=\{d\}$ where $w$ is coloured
  $d$ in $G_1$.  Let $L_2(x):=L(x)$ for every other vertex $x$ in
  $G_2$.  Then $(G_2,\{v,w\},B_2,L_2)$ is not a counterexample (since
  $\{v,w\}\subseteq B_2$), and $G_2$ is $L_2$-colourable. Hence $G$ is
  $L$-colourable. Now assume that $G$ contains no such cut-set
  $\{v,w\}$.

  \textbf{\boldmath Case 7. Some vertex $v\in B$ has degree at least
    $3$ in $G[B]$:} Let $x,y,z$ be three neighbours of $v$ in $B$. If
  $G-v$ contains a $K_3$-minor rooted at $x,y,z$, then adding $v$ and
  $\alpha$ gives a $K_5$-minor in $\plus{G}{B}$, which contradicts the
  assumption that $B$ is an $\M$-boundary of $G$. Thus $G-v$ contains
  no $K_3$-minor rooted at $x,y,z$. By \lemref{GeneralRooted}, for
  some vertex $w$ in $G-v$, all the vertices in
  $\{x,y,z\}\setminus\{w\}$ are in distinct components of
  $G-\{v,w\}$. Thus $\{v,w\}$ is a cut-set satisfying Case 6 or 7.
  Now assume that $G[B]$ has maximum degree at most $2$.

  \textbf{\boldmath Case 8. $G[A]$ is a component of $G[B]$:} Choose
  $v\in A$. Let $G':=G-v$ and $A':=A\setminus\{v\}$ and
  $B':=(B\setminus\{v\})\cup N_G(v)$. By \lemref{Contract}, $B'$ is an
  $\M$-boundary of $G'$. Let $L'(u):=L(u)\setminus L(v)$ for each
  vertex $u\in N_G(v)\setminus B$.  Since $|L(v)|=1$ and $v$ has no
  neighbour in $B\setminus A$, we have $|L'(x)|\geq3$ for each $x\in
  B'$.  Let $L'(x):=L(x)$ for every other vertex $x$.  Then
  $(G',A',B',L')$ is not a counterexample, and $G'$ is
  $L'$-colourable. Assign $v$ the colour in $L(v)$. This colour is not
  in $L'(u)$ for each $u\in N_G(v)$. Thus $G$ is $L$-colourable.

  \textbf{\boldmath Case 9. $G[A]$ is not a component of $G[B]$:}
  Choose $v\in B\setminus A$ adjacent to some vertex $p\in A$. Since
  $G[B]$ has maximum degree at most 2, $v$ has at most one other
  neighbour in $B$; let $w$ be this neighbour (if it exists). Let
  $G':=G-v$ and $B':=(B\setminus\{v\})\cup N_G(v)$. By
  \lemref{Contract}, $B'$ is an $\M$-boundary of $G'$. Let $c,d$ be
  distinct colours in $L(v)\setminus L(p)$. Let
  $L'(u):=L(u)\setminus\{c,d\}$ for each vertex $u\in N_G(v)\setminus
  B$; thus $|L'(u)|\geq 5-2=3$. Let $L'(x):=L(x)$ for every other
  vertex $x$. Then $(G',A,B',L')$ is not a counterexample, and $G'$ is
  $L'$-colourable. Assign $v$ colour $c$ or $d$ different from the
  colour assigned to $w$ (if $w$ exists). Hence $G$ is $L$-colourable.
\end{proof}

%%%%%%%%%%%%%%%%%%%%%%%%%%%%%%%%%%%%%%%%%%%%%%%%%%%%%%%%%%%%%
% \bibliographystyle{myNatbibStyle}
% \bibliography{myBibliography,myConferences}
%%%%%%%%%%%%%%%%%%%%%%%%%%%%%%%%%%%%%%%%%%%%%%%%%%%%%%%%%%%%%
\def\cprime{$'$}
\def\soft#1{\leavevmode\setbox0=\hbox{h}\dimen7=\ht0\advance \dimen7
  by-1ex\relax\if t#1\relax\rlap{\raise.6\dimen7
    \hbox{\kern.3ex\char'47}}#1\relax\else\if T#1\relax
  \rlap{\raise.5\dimen7\hbox{\kern1.3ex\char'47}}#1\relax \else\if
  d#1\relax\rlap{\raise.5\dimen7\hbox{\kern.9ex
      \char'47}}#1\relax\else\if D#1\relax\rlap{\raise.5\dimen7
    \hbox{\kern1.4ex\char'47}}#1\relax\else\if l#1\relax
  \rlap{\raise.5\dimen7\hbox{\kern.4ex\char'47}}#1\relax \else\if
  L#1\relax\rlap{\raise.5\dimen7\hbox{\kern.7ex
      \char'47}}#1\relax\else\message{accent \string\soft \space #1
    not defined!}#1\relax\fi\fi\fi\fi\fi\fi}
\def\Dbar{\leavevmode\lower.6ex\hbox to
  0pt{\hskip-.23ex\accent"16\hss}D}

\end{document}